\documentclass[reqno, 12pt]{amsart}
\usepackage{amsmath,amsthm,amsfonts,amssymb}
\date{}
\newtheorem{theorem}{Theorem}[section]
\newtheorem{lemma}[theorem]{Lemma}

\newtheorem{remark}[theorem]{Remark}
\newtheorem{proposition}[theorem]{Proposition}
\newtheorem{example}[theorem]{Example}
\numberwithin{equation}{section}

\begin{document}

\centerline{{\large\bf Convergence  in variation of solutions of nonlinear}}

\centerline{{\large\bf Fokker--Planck--Kolmogorov equations to stationary measures}}

\vskip .1in

\centerline{\bf V.I. Bogachev$^{a}$\footnote{corresponding author.
{\it E-mail addresses}: vibogach@mail.ru (V. Bogachev),
roeckner@math.uni-bielefeld.de (M. R\"ockner), starticle@mail.ru (S. Shaposhnikov)},
M. R\"ockner$^{b}$,
S.V. Shaposhnikov$^{a}$}

\vspace*{0.2cm}

\noindent
\small{$^{a}${\it Lomonosov Moscow State University and National Research University Higher School of Economics, Moscow, Russia}

\noindent
$^{b}${\it Fakult\"at f\"ur Mathematik, Universit\"at Bielefeld,
D-33501 Bielefeld, Germany}
}

\vspace*{0.3cm}

{\bf Abstract}
We study convergence in variation of probability solutions of nonlinear
Fokker--Planck--Kolmogorov equations to stationary solutions.
We obtain sufficient conditions for the exponential convergence of solutions to  the stationary solution
in case of coefficients that can have an arbitrary growth at infinity and depend
on the solutions through convolutions with unbounded discontinuous kernels.
In addition, we study a more difficult case where the nonlinear equation has several stationary solutions and convergence
to a  stationary solution depends on initial data. Finally, we obtain sufficient conditions for solvability of
nonlinear Fokker--Planck--Kolmogorov equations.

\vskip .1in



\section{Introduction}

We consider the Cauchy problem for a nonlinear Fokker--Planck--Kolmogorov
equation
\begin{equation}\label{eq1A}
\partial_t\mu_t=\sum_{1\le i, j\le d}\partial_{x_i}
\partial_{x_j}\bigl(a^{ij}(x)\mu_t\bigr)-
\sum_{i=1}^d\partial_{x_i}\bigl(b^i(x, \mu_t)\mu_t\bigr),
\quad \mu_0=\nu,
\end{equation}
in which the nonlinearity originates from the dependence of the drift term $b$
on the unknown solution.
This equation is understood in the sense of the integral identity
\begin{multline}\label{eq1Aint}
\int_{\mathbb{R}^d} \varphi(x)\, \mu_t(dx) -
\int_{\mathbb{R}^d} \varphi(x)\, \nu(dx)
\\
=\int_0^t \int_{\mathbb{R}^d}
[{\rm trace}\, (A(x)D_x^2\varphi(x))+\langle b(x,\mu_s),\nabla
\varphi(x)\rangle]
  \, \mu_s(dx)\, ds ,\quad \varphi\in C_0^\infty(\mathbb{R}^d).
\end{multline}
Solutions to the stationary equation
\begin{equation}\label{eq1AS}
\sum_{1\le i, j\le d}\partial_{x_i}\partial_{x_j}\bigl(a^{ij}(x)\mu_t\bigr)-
\sum_{i=1}^d\partial_{x_i}\bigl(b^i(x, \mu_t)\mu_t\bigr)=0
\end{equation}
are defined similarly by means of the integral identity
\begin{equation}\label{eq1ASint}
\int_{\mathbb{R}^d}
[{\rm trace}\, (A(x)D_x^2\varphi(x))+\langle b(x,\mu),\nabla \varphi(x)\rangle]
  \, \mu(dx)=0,\quad \varphi\in C_0^\infty(\mathbb{R}^d).
\end{equation}
For  a recent detailed presentation of the  theory
of linear Fokker--Planck--Kolmogorov  equations, see~\cite{BKRSH-book},
where also some comments on nonlinear problems can be found.

Our main result (Theorem~\ref{th1}) gives sufficient conditions for the existence of a
stationary probability solution
$\mu$ and the exponential convergence to this stationary solution
in total variation norm. To be more precise,
we obtain a bound
$$
\|W\cdot(\mu_t-\mu)\|_{TV}\le \alpha_1 e^{-\alpha_2t}
$$
with a suitable growing function $W$ (so that the left-hand side dominates
the usual total variation norm).
Informally, our results are of the following nature:
the drift term $b$ in the nonlinear equation
depends on a parameter $\varepsilon\ge 0$ such that for $\varepsilon=0$ the equation
has certain nice properties, for  example, becomes linear (actually, the situation is more general)
with reasonable properties, then for $\varepsilon$ small enough both
the original nonlinear equation and the stationary equation are solvable and
we have exponential convergence in  total variation norm with a weight.

First of all we consider the following very typical example demonstrating  phenomena
arising  in the study of convergence of solutions of a nonlinear
Fokker--Planck--Kolmogorov equation to the stationary distribution.

\begin{example}\label{ex1.1}
\rm
Let $d=1$, $A=I$ and $b(x, \mu)=-x+\varepsilon B(\mu)$, where
$$
B(\mu)=\int_{\mathbb{R}}x\, \mu(dx).
$$
In case $\varepsilon<1$ the unique
solution of the stationary  equation is the standard Gaussian measure $\mu$ (see, e.g., \cite{BKSH17}).
One can show that the transition probabilities $\mu_t$ forming
the solution   to the  Cauchy problem (\ref{eq1A}),
for every initial condition~$\nu$ (with a finite first moment),
converge exponentially to the stationary measure. This is discussed in Remark~\ref{r3.7}
along with the case $\varepsilon>1$.
If $\varepsilon=1$, then  every measure $\mu$ given by a density
$$
\varrho_a(x)=\frac{1}{\sqrt{2\pi}}\exp\bigl(-|x-a|^2/2\bigr),
\quad a\in\mathbb{R}^d,
$$
satisfies the stationary equation.
It is readily seen that the measures $\mu_t$ converge to that stationary measure which
has the same mean as~$\nu$.
Indeed, in the case under consideration the mean of $\mu_t$ does not depend on time
and coincides with the mean of $\nu$.
Therefore, if the mean of $\nu$ coincides with that of~$\mu$, then the mean
of $\mu_t$ coincides with the mean of~$\mu$, i.e.,
$B(\mu_t)=B(\mu)$, and the measures $\mu_t$ satisfy the linear
Fokker--Planck--Kolmogorov equation corresponding to the Ornstein--Uhlenbeck type operator,
for which convergence to the solution of the stationary  equation is well known.
\end{example}

Thus, already in this very simple one-dimensional example we see
that convergence to the stationary  distribution
depends not only on the form of  the nonlinearity, but also on the initial condition.
Moreover, an important  role is  played by certain quantities invariant along the
trajectories of solutions to the  Fokker--Planck--Kolmogorov equation.
Note that the existence of a  stationary solution  and convergence
to it are not stable under small perturbations of the coefficients. For example, if $b(x, \mu)
=-x+\delta+\varepsilon B(\mu)$ with arbitrarily small $\delta>0$,
then for $\varepsilon=1$ there are no stationary solutions,
because  for the solution  $\mu$ we must have the equality $(1-\varepsilon)B(\mu) =\delta$.

In this paper we show that the picture described in this special example
takes place in a very general situation.
Certainly, in the general case it becomes difficult to take into account specific features
of concrete equations, but  some typical cases can be singled out.

In the example considered  above the coefficient  $b(x, \mu)$ has the form $b_0(x)
+\varepsilon b_1(x, \mu)$.
Con\-ver\-gence of solutions to the  Cauchy  problem for nonlinear Fokker--Planck--Kolmogorov equations
 with drift  coefficients of such a form have been studied in the paper
\cite{OB}, where it has been shown that convergence to the stationary distribution takes place
in case of a sufficiently small number $\varepsilon$,
provided that the coefficients are Lipschitz in $x$ and Lipschitz
in $\mu$ with respect to the Kantorovich metric.
In addition, the term $b_1$ has been assumed to be globally bounded.
In the paper \cite{Ebr} a similar result has been obtained with the aid of the method
of coupling in the case where
$$
\langle b_0(x)-b_0(y), x-y\rangle\le -\kappa(|x-y|)|x-y|,
$$
$$
b_1(x, \mu)=\int_{\mathbb{R}^d} K(x, y)\,d\mu, \quad |K(x, y)-K(x', y')|\le C(|x-x'|+|y-y'|),
$$
i.e., the global boundedness of $b_1$ has been weakened by means of the monotonicity condition
for $b_0$.
The smallness of the parameter  $\varepsilon$ is important not only for convergence, but also
for the existence of a  stationary distribution, which is seen from the example  above.
However, the situation is actually more complicated  (and this is also seen from the example above)
and convergence depends on the initial condition
and the stationary measure. An important role in determining conditions under which one has convergence is played
by certain quantities that are invariant or unboundedly increasing along trajectories of the
  Fokker--Planck--Kolmogorov equation (in our example such a quantity
is the centre of mass). Finding such quantities enables one to single out classes of initial conditions
for which one has convergence to the stationary distribution or to prove
that there is no such convergence. In the case of a nondegenerate diffusion  the
assumption of Lipschitzness of $b(x, \mu)$ in $\mu$ with respect to the
metric employed in \cite{OB} and \cite{Ebr} becomes superfluous,
because any solution possesses a density  and it is more natural to replace the Kantorovich metric by the
total variation distance.

In this paper we study convergence in variation.
The next two examples illustrate our main results.
Let $m\ge 1$ and
$$
b_{\varepsilon}(x, \mu)=b_0(x)+\varepsilon \int_{\mathbb{R}^d} K(x, y)\,\mu(dy),
$$
where for all $x,y\in\mathbb{R}^d$,
$$
\langle b_0(x), x\rangle\le c_1-c_2|x|^2, \quad \langle K(x, y), x\rangle \le c_3+c_3|x|^2, \quad C_3<c_2,
$$
$$
|b_0(x)|\le c_4+c_4|x|^m, \quad |K(x, y)|\le c_5(1+|x|^m)(1+|y|^m).
$$
Then, there is $\varepsilon_0>0$ such that for all $\varepsilon\in(0, \varepsilon_0)$ there exists a solution
 $\mu$ to the stationary Fokker--Planck--Kolmogorov equation (\ref{eq1AS})
with the coefficients $A=I$ and $b_{\varepsilon}(x, \mu)$.
Moreover, for every probability measure $\nu$ such that  $(1+|x|)^{2m+1}\in L^1(\nu)$, the solutions $\mu_t$
to the Cauchy problem  (\ref{eq1A}) with initial data $\nu$
converge to $\mu$ as $t\to+\infty$ and
$$
\|(1+|x|^m)(\mu_t-\mu)\|_{TV}\le \alpha_1e^{-\alpha_2t}, \quad \alpha_1, \alpha_2>0.
$$
Our second example concerns the case where the stationary equation has several solutions.
Let $d=2$, $x=(x_1, x_2)$, $y=(y_1, y_2)$
and $b_{\varepsilon}(x, \mu)=(b_{\varepsilon}^1(x, \mu), b_{\varepsilon}^2(x, \mu))$, where
$$
b_{\varepsilon}^1(x, \mu)=-2x_1+\int_{\mathbb{R}^2} (y_1+y_2)\,\mu(dy)+\varepsilon\int_{\mathbb{R}^2} H(x, y)\,\mu(dy),
$$
$$
b_{\varepsilon}^2(x, \mu)=-2x_2+\int_{\mathbb{R}^2} (y_1+y_2)\,\mu(dy)-\varepsilon\int_{\mathbb{R}^2} H(x, y)\,\mu(dy),
$$
with some bounded Borel function
$H\colon\mathbb{R}\times \mathbb{R}\to\mathbb{R}$.

Then, for every number $Q_0>0$, there is a number $\varepsilon_0>0$, depending
only on $Q_0$, such that for any $\varepsilon\in(0, \varepsilon_0)$ and
$Q\in (-Q_0, Q_0)$ there exists a solution $\mu$ to the stationary
Fokker--Planck--Kolmogorov equation (\ref{eq1AS})
with the coefficients $A=I$ and $b_{\varepsilon}(x, \mu)$ for which
$$
\int_{\mathbb{R}^2} (y_1+y_2)\,\mu(dy)=Q.
$$
Moreover, for every probability measure $\nu$ such that $|x|^2\in L^1(\nu)$ and
$$
\int_{\mathbb{R}^2} (y_1+y_2)\,\nu(dy)=Q,
$$
the solutions  $\mu_t$ to the Cauchy problem (\ref{eq1A}) with initial data $\nu$
converge to $\mu$ as $t\to+\infty$ and
$$
\|(1+|x|)(\mu_t-\mu)\|_{TV}\le \alpha_1e^{-\alpha_2t}, \quad \alpha_1, \alpha_2>0.
$$
More general conditions and examples are discussed after Theorems \ref{th1} and \ref{th2}.

Note that it is often simpler,
and in the case of a degenerate diffusion matrix more natural, to consider convergence in the Kantorovich metric
(see Remark~\ref{remK} below).
Results of this sort  for non-gradient drift coefficients were apparently first
obtained  in~\cite{AHMED93} and have been recently generalized in  \cite{Ebr}, \cite{Ver03}, and \cite{Wang17}.
See  also the related, but more special papers \cite{Benachour1} and \cite{Benachour2}.
The gradient case, where $b=\nabla V$, has been studied  in many papers,
starting from \cite{Daw}, \cite{Tam1}, \cite{Tam2} and further studied in many papers on
the theory of gradient flows (see \cite{AMBSG},  \cite{BoGenGu}, \cite{CarMacVil}, and \cite{GDFL}).
In the theory of gradient flows an important  role is played by the Kantorovich
$2$-metric and the geometry of the space of probability measures connected with this metric.

There is a vast literature devoted to nonlinear  Fokker--Planck--Kolmogorov equations
 (see, e.g., \cite{Frank}).
It should be emphasized  that in this paper  we study equations with
nonlocal nonlinearities (of the type of the so-called McKean--Vlasov equations),
the investigation of which was initiated  in the well-known papers
\cite{Kac}, \cite{McKean1}, \cite{McKean2}, \cite{Funaki} and continued by many researchers.
  This circumstance explains the character of our assumptions about the drift, which are quite natural
  for such nonlinearities. For instance, the continuity of $b(\mu,x)$ in $\mu$ in total variation norm
   holds when  $b(\mu,x)$  depends on $\mu$ through convolution, but not when $b(\mu,x)$ depends
  on the value of the density of $\mu$ at~$x$.
Existence and uniqueness of solutions  and properties of the distributions of stochastic McKean--Vlasov
equations (distributions of such equations satisfy nonlinear Fokker--Planck--Kolmogorov equations)
are discussed, e.g., in  \cite{Ver16} and \cite{Wang17}. In particular,
in \cite{Wang17}, sufficient conditions
(monotonicity of the coefficient~$b$, as in \cite{Ebr}) for the
existence and uniqueness of solutions are given and
convergence in the Kantorovich $2$-metric to the stationary distribution is shown.

Existence and uniqueness of solutions to nonlinear  Fokker--Planck--Kolmogorov  equations
with irregular and rapidly growing
coefficients have been discussed in the recent papers \cite{ManRomShap} and \cite{ManShap},
which also contain some examples of non-uniqueness. The papers \cite{JF16} and \cite{Man} develop
an approach to nonlinear equations based on  estimates of
distances between solutions to linear   Fokker--Planck--Kolmogorov equations
with different diffusion matrices and different drift coefficients.
Analogous questions for nonlinear stationary
Fokker--Planck--Kolmogorov equations are studied in  \cite{BKSH14}, \cite{BKSH17}, and \cite{Ton},
 where the existence of stationary solutions is proved with the
aid of  fixed point theorems applied to
the nonlinear mapping that maps a probability measure $\sigma$ to the solution  $\mu$
of the linear equation with the drift coefficient $b(x, \sigma)$. The phenomenon of nonuniqueness
of a stationary measure is investigated in \cite{HT10}, where certain explicit nonlinear
expressions for stationary measures are written out in the gradient case.
In this paper  we also obtain some generalizations of
existence and uniqueness results for stationary and parabolic Fokker--Planck--Kolmogorov equations.

The problem of convergence to the stationary  measure for a linear
Fokker--Planck--Kolmogorov equation has been thoroughly studied, and one can single out the following three
approaches:

1) the approach based on the Harris theorem or the
Meyn--Tweedie approach with Lyapunov functions (see, e.g., \cite{DowMeinTvid} and \cite{HM}),

2) the approach based on entropy estimates and Poincar\'e and Sobolev inequalities
(see, e.g., \cite{Ar1}, \cite{AMTU}, \cite{BakryG}, \cite{Boll}, \cite{Cat},
\cite{RW}, and~\cite{Wang-book}),

3) the probabilistic approach based on coupling (see, e.g., \cite{Ebr0},
\cite{Ebr}, \cite{ChLi}, and \cite{LuoW}).

In this paper  we employ the first approach and for verification of the conditions of the Harris theorem
we use certain estimates for transition probabilities from \cite[Chapter~8]{BKRSH-book}.
This enables us to substantially weaken the assumptions about the regularity
of coefficients, but, on the  other hand, some weak points of the Meyn--Tweedie approach
connected with a complicated dependence of constants on the  coefficients of the equation remain
also in our case.

Throughout  we assume  that the matrix $A(x)=(a^{ij}(x))_{1\le i, j\le d}$
is symmetric and there exist numbers $K_1>0$ and $K_2>0$ such that
$$
K_1^{-1}I\le A(x)\le K_1I, \quad |A(x)-A(y)|\le K_2|x-y|.
$$

Let $V\in C^2(\mathbb{R}^d)$, $V\ge 1$ and $\lim\limits_{|x|\to+\infty}V(x) =+\infty$.
Let $\mathcal{P}_V(\mathbb{R}^d)$ denote the space of all probability measures
$\mu$ on $\mathbb{R}^d$ such that
$$
\int_{\mathbb{R}^d} V\,d\mu<\infty.
$$
Set
$$
W(x)=V(x)^{\gamma}, \quad \gamma\in(0, 1/2].
$$
Typical examples are  $V(x)=1+|x|^{2m}$.
We recall that the total variation norm of a finite (possibly, signed)
measure $\sigma$ is defined by
$$
\|\sigma\|_{TV}=|\sigma|(\mathbb{R}^d),
$$
where $|\sigma|=\sigma^{+}+\sigma^{-}$ and $\sigma=\sigma^{+}-\sigma^{-}$
is the Hahn decomposition
into the difference of mutually singular nonnegative measures.
The symbol $W\cdot \mu$ denotes the measure given by the density $W$
with respect to the measure~$\mu$.
Set
$$
\|\mu\|_{W}=\|W\cdot\mu\|_{TV}.
$$

Suppose that for every $\mu\in \mathcal{P}_V(\mathbb{R}^d)$ we have
 a Borel vector field $b(x, \mu)=(b^i(x, \mu))_{1\le i\le d}$
on $\mathbb{R}^d$  such that  there exists a number $C(\mu)$ for which
$$|b(x, \mu)|\le C(\mu)V(x)^{1-\gamma}.$$
It will be assumed below that $b$ satisfies certain additional conditions.

We say that a family $\{\mu_t\}_{t\in [0,T]}$ of probability
measures $\mu_t\in\mathcal{P}_V(\mathbb{R}^d)$ satisfies the
Cauchy problem (\ref{eq1A}) on $[0,T]$, where $T>0$ is fixed, if
 equality (\ref{eq1Aint}) holds.
A measure $\mu\in\mathcal{P}_V(\mathbb{R}^d)$ is called a solution to the
stationary  equation  (\ref{eq1AS}) if  equality (\ref{eq1ASint}) is fulfilled.

Set
$$
L_{\mu}\varphi(x)={\rm trace}(A(x)D^2\varphi(x))+\langle b(x, \mu), \nabla\varphi(x) \rangle.
$$
Throughout for the stationary  equation (\ref{eq1AS}) and for the parabolic
equation  (\ref{eq1A})
we use the shortened equalities $L_{\mu}^{*}\mu=0$ and
$\partial_t\mu_t=L_{\mu_t}^{*}\mu_t$.
Similarly we write  linear equations $L_{\sigma}^{*}\mu=0$
and $\partial_t\mu_t=L_{\sigma}^{*}\mu_t$
with the coefficient $b(x, \sigma)$. Solutions to linear
Fokker--Planck--Kolmogorov equations are defined
precisely as in the case of general nonlinear equations by means of integral equalities
of the form (\ref{eq1Aint}) and (\ref{eq1ASint}).

This paper  consists of the introduction  and three sections.
In Section~2 we discuss some classes of functions $\psi$ on $\mathbb{R}^d$ such
that the integral of $\psi$ against $\mu_t$ for solutions $\{\mu_t\}$
to the  Cauchy problem (\ref{eq1A})
is constant or equals a constant multiplied by a function
of the form $\exp(\lambda t)$.

With the aid of such functions one can formulate simple tests to show that
convergence to stationary solutions fails to hold. In addition, if we know that
$$\int_{\mathbb{R}^d} \psi\,d\mu_t\equiv \int\psi\,d\nu
$$
with $\mu_t$ satisfying our Cauchy problem
and
$$
b(x, \mu)=b_0(x)+\int_{\mathbb{R}^d} \psi\,d\mu,
$$
then
$$
b(\mu_t)=b_0(x)+const
$$
i.e.,  along solutions the drift depends only on $x$ and actually is independent
of~$\mu_t$.

At the end of Section 2 we formulate our main
conditions on the  coefficients of the equation. In Section~3 we formulate
and prove the main results of the paper. The first main result
(Theorem~\ref{th1})
enables us to determine by the initial condition and the stationary solution whether there is
convergence of solutions of the  Cauchy problem (\ref{eq1A}) with this initial condition to the
stationary solution. The second main result
(Theorem~\ref{th2}) gives sufficient conditions under which a stationary
 solution exists and for every initial condition (having a  finite moment of a suitable order)
the solutions of the  Cauchy problem converge to  this stationary solution.
In Section~4 we discuss
conditions for the existence of solutions to the stationary equation and the Cauchy problem.

\section{Invariant  and subinvariant functions}

Here we consider certain conservation laws for solutions.

For a function $W$ as above,
 let  $I_0^W$ denote the set of  functions $\psi\in C^2(\mathbb{R}^d)$ such that
\begin{equation}\label{inves}
\sup_{x}\Bigl(|\psi(x)|+|\nabla\psi(x)|+|D^2\psi(x)|\Bigr)W(x)^{-1}<\infty
\end{equation}
and for every measure $\mu\in\mathcal{P}_V(\mathbb{R}^d)$ we have
\begin{equation}\label{inveq}
\int_{\mathbb{R}^d}L_{\mu}\psi\,d\mu=0.
\end{equation}

It is clear that $I_0^W$ is a linear space containing $1$.
Further  for brevity we use the notation
$$
\mu(\psi):=\int_{\mathbb{R}^d}\psi\,d\mu.
$$

\begin{proposition}\label{pr00}
\, {\rm (i)} \, If $\psi\in I_0^W$ and $\{\mu_t\}$ is a solution to the
Cauchy problem {\rm (\ref{eq1A})}
with the initial condition $\nu$, then $\mu_t(\psi)=\nu(\psi)$.

\, {\rm (ii)} \, If $\mu\in\mathcal{P}_V(\mathbb{R}^d)$
is a solution to the stationary equation {\rm (\ref{eq1AS})} and
$$\nu(\psi)\neq \mu(\psi)$$
for some function $\psi\in I_0^W$, then
the solutions $\mu_t$ to the  Cauchy problem {\rm (\ref{eq1A})} with the initial condition
$\nu$ do not converge to $\mu$ with respect to the norm $\| \,\cdot\, \|_{W}$.
\end{proposition}
\begin{proof}
Assertion  (ii) follows from (i). For justifying (i) it suffices to observe that
$$
\mu_t(\psi)-\nu(\psi)=\int_0^t\int_{\mathbb{R}^d}L_{\mu_s}\psi\,d\mu_s\,ds=0,
$$
which follows from the equation.
\end{proof}

Let us consider an important example where  $A=I$ and
$$
b(x, \mu)=-\int_{\mathbb{R}^d}K(x, y)\,\mu(dy),
$$
where $K$ is a vector-valued mapping.

\begin{proposition}\label{pr01}
A function  $\psi$ satisfying {\rm (\ref{inves})} belongs to $I_0^W$ if and only if
for all $x, y$ we have
$$
\Delta\psi(x)+\Delta\psi(y)- \langle K(x, y),\nabla\psi(x)\rangle -\langle K(y, x),\nabla\psi(y)\rangle=0.
$$
In particular, if $\psi\in I_0^W$, then $\Delta\psi(x)-\langle K(x, x),\nabla\psi(x)\rangle=0$.
Moreover, if
$$(Q,K(x, y))=-(Q,K(y, x))$$
for some constant vector $Q$
and $W(x)$ is growing not more slowly than $|x|$, then $I_0^W$ contains all
functions of the form $\psi(x)=(Q,x)+g$, where $g$ is a constant number.
\end{proposition}
\begin{proof}
We observe that (\ref{inveq}) is equivalent to the equality
$$
\int_{\mathbb{R}^d\times\mathbb{R}^d}
\Bigl[\Delta\psi(x)+\Delta\psi(y)-\langle K(x, y),\nabla\psi(x)\rangle-\langle K(y, x),\nabla\psi(y)\rangle
\Bigr]\, \mu\otimes\mu(dx\, dy)=0,
$$
which  holds for every probability measure $\mu\in\mathcal{P}_V(\mathbb{R}^d)$
if and only if the expression under the integral sign is skew symmetric in $x$ and $y$.
Since we have a symmetric function there, it must vanish.
\end{proof}

\begin{example}\rm
\, (i) \, Let $d=1$. The space of solutions to the equation  $\psi''-K(x, x)\psi'=0$
 is the linear span of $1$ and the primitive of the function
$$
\exp\Bigl(\int_0^xK(s, s)\,ds\Bigr).
$$
A nonconstant function $\psi$  satisfies the condition of Proposition \ref{pr01}
if and only if for all $x, y$ we have
$$
K(x, x)+K(y, y)=K(x, y)+K(y, x).
$$
The latter relation is satisfied, for example, for the functions $K(x, y)=H(x-y)$,
 where $H$ is an odd function. In this case  $I_0^W$ is the linear span of  $1$ and $x$.

\, (ii) \, Let $d\ge 1$ and
$$
K(x, y)=-Rx+\langle v, y\rangle h+H(x, y),
$$
where $R$ is a constant matrix, $v$ and $h$ are constant vectors
and
$$R^{*}v=\lambda v, \quad \langle v, h\rangle=\lambda, \quad \langle H(x, y), v\rangle=0.$$
Then the function
$x\to \langle v, x\rangle$
belongs to $I_0^W$. Indeed, we have
$$
\langle v, K(x, y)\rangle=-\lambda\langle v, x\rangle+\lambda\langle v, y\rangle=-\langle v, K(y, x)\rangle.
$$
\end{example}

Let  $I_{+}^W$ denote  the set of all functions $\psi\in C^2(\mathbb{R}^d)$ such  that
$\psi$ satisfies  condition (\ref{inves}) and there exists a number
$\lambda=\lambda(\psi)>0$ such that
\begin{equation}\label{subinv}
\int_{\mathbb{R}^d}L_{\mu}\psi\,d\mu=\lambda\int_{\mathbb{R}^d}\psi\,d\mu
\quad \forall\, \mu\in\mathcal{P}_V(\mathbb{R}^d).
\end{equation}

\begin{proposition}\label{pr02}
\, {\rm (i)} \, If $\psi\in I_{+}^W$ and $\mu_t$ is a solution to the  Cauchy problem {\rm (\ref{eq1A})}
with the  initial condition $\nu$, then $\mu_t(\psi)=\nu(\psi)e^{\lambda(\psi)t}$.

\, {\rm (ii)} \, If $\mu\in\mathcal{P}_V(\mathbb{R}^d)$ is a solution to the
stationary equation  {\rm (\ref{eq1AS})} and $\psi\in I_{+}^W$, then $\mu(\psi)=0$.

\, {\rm (iii)} \, If $\nu(\psi)\neq 0$ for some $\psi\in I_{+}^W$, then the solutions $\mu_t$
to the  Cauchy problem {\rm (\ref{eq1A})}
do not converge to the stationary solution with respect to the norm $\| \,\cdot\, \|_W$.
\end{proposition}
\begin{proof}
Assertion  (ii) is obvious. Assertion (iii) follows from~(i), and (i) is deduced from the equality
$$
\mu_t(\psi)-\nu(\psi)=\int_0^t\int_{\mathbb{R}^d}L_{\mu_s}\psi\,d\mu_s\,ds
=\lambda\int_0^t\mu_s(\psi)\,ds,
$$
satisfied by the solution.
\end{proof}

As above, let us consider the case where $A=I$ and
$$
b(x, \mu)=-\int_{\mathbb{R}^d}K(x, y)\,\mu(dy)
$$
with a vector-valued mapping $K$.

\begin{proposition}\label{pr03}
A function  $\psi$ satisfying {\rm (\ref{inves})} belongs to
$I_{+}^W$ if and only if for some $\lambda>0$ and for all $x, y$ we have
$$
\Delta\psi(x)+\Delta\psi(y)-(K(x, y),\nabla\psi(x))-(K(y, x),\nabla\psi(y))=\lambda(\psi(x)+\psi(y)).
$$
In particular, if $\psi\in I_{+}^W$, then $\Delta\psi(x)-(K(x, x),\nabla\psi(x))=\lambda\psi(x)$.
\end{proposition}
\begin{proof}
The same reasoning as in the case of $I_0^W$ works.
\end{proof}

Let $d=1$ and $K(x, x)=-qx$, where $q$ is a positive constant.
Then the equation on the function $\psi$ takes the form
$$
\psi''+qx\psi'=\lambda\psi.
$$
If $\lambda=q$, then $\psi(x)=x$ is a solution.
If $K(x, x)=0$, then the equation takes the form $\psi''=\lambda\psi$ and
 linear combinations of the exponents  $e^{\sqrt{\lambda}x}$
and $e^{-\sqrt{\lambda}x}$ are all solutions.

\begin{proposition}\label{pr04}
Suppose that for some function $\psi\in I_0^W$ there exists a
continuous function $h$ such that
$\sup_x|h(x)|/V(x)<\infty$ and for every probability measure
 $\sigma\in\mathcal{P}_V(\mathbb{R}^d)$ we have
$$
L_{\sigma}\psi(x)=C_1(\sigma)h(x)+C_2(\sigma), \quad C_1(\sigma)\neq 0
$$
with some numbers $C_1(\sigma)$ and $C_2(\sigma)$.
Suppose that $\mu\in\mathcal{P}_V(\mathbb{R}^d)$
satisfies the stationary equation $L^{*}_{\sigma}\mu=0$.
Then $\mu(h)=\sigma(h)$.
The analogous assertion is true if $\psi\in I_{+}^W$ and $\sigma(\psi)=0$.
\end{proposition}
\begin{proof}
By the definition of $I_0^W$ and the fact that $\mu$ is a solution to the
stationary  equation  $L^{*}_{\sigma}\mu=0$ we have the equalities
$$
C_1(\sigma)\int_{\mathbb{R}^d}h\,d\sigma+C_2(\sigma)=
\int_{\mathbb{R}^d}L_{\sigma}\psi\,d\sigma=0
=\int_{\mathbb{R}^d}L_{\sigma}\psi\,d\mu=C_1(\sigma)\int_{\mathbb{R}^d}h\,
d\mu+C_2(\sigma).
$$
Since $C_1(\sigma)\neq 0$, we obtain $\mu(h)=\sigma(h)$.
\end{proof}

\begin{example}\label{exa1}
\rm
\, (i) \, Let $d=1$, $A=1$ and
$$
b(x, \mu)=f(x)-\int_{\mathbb{R}^d}f(y)\,\mu(dy),
$$
where $f$ is a reasonable function. Then
the function $\psi(x)=x$ belongs to $I_0^W$ and for every $\sigma$ one has
$$
L_{\sigma}x=f(x)-\int_{\mathbb{R}^d}f(y)\,\sigma(dy)=f(x)+C_2(\sigma).
$$
Therefore, for the stationary solution  $\mu$ of the equation with the operator
$L_{\sigma}$ the equality $\mu(f)=\sigma(f)$ holds.

\, (ii) \, Let $d\ge 1$, $A=I$,
$$b(x, \mu)=\int_{\mathbb{R}^d} K(x, y)\,\mu(dy)$$
and
$$
K(x, y)=-Rx+\langle v, y\rangle h+H(x, y),
$$
where $R$ is a constant matrix, $v$ and $h$ are constant vectors and
$$R^{*}v=\lambda v, \quad \lambda\neq 0, \quad \langle v, h\rangle=\lambda, \quad \langle H(x, y), v\rangle=0.$$
Then the function $x\to \langle v, x\rangle$ belongs to $I_0^W$ and
$$
L_{\sigma}\langle v, x\rangle=-\lambda\langle v, x\rangle+\lambda \int_{\mathbb{R}^d} \langle v, x\rangle\,\sigma(dy).
$$
Therefore, for the stationary solution $\mu$ to the equation
with the operator $L_{\sigma}$ we have that $\mu(h)=\sigma(h)$ with $h(x)=\langle v, x\rangle$.
In particular, if $d=2$, $x=(x_1, x_2)$, $y=(y_1, y_2)$ and
$$
K_1(x, y)=-2x_1+(y_1+y_2)+H(x, y), \quad K_2(x, y)=-2x_2+(y_1+y_2)-H(x, y),
$$
then $\mu(x_1+x_2)=\sigma(x_1+x_2)$. Here $R=2I$, $v=(1, 1)$ and $h=(1, 1)$.
\end{example}

Proposition  \ref{pr04} differs from Propositions \ref{pr00} and \ref{pr02}
in which we studied the dynamics of certain quantities
along trajectories of solutions.
Proposition  \ref{pr04} will play the key role in constructing stationary solutions
to nonlinear Fokker--Planck--Kolmogorov equations in Section~4
 (see Proposition \ref{pr1}).
Note that the observations above are rather rough and in special
situations more refined considerations are possible (see, for example, \cite{HT10}),
but it seems reasonable to begin the study of convergence of solutions of the Cauchy
problem  to the solution of the stationary equation from finding
quantities a priori invariant  or subinvariant along trajectories of solutions.

Closing this section we formulate our conditions on the coefficients
in terms of the sets $I_0^W$ and $I_{+}^W$.
It is reasonable  (with regards towards convergence) to consider only
measures $\mu\in\mathcal{P}_V(\mathbb{R}^d)$
such that $\mu(\psi)=0$ for every function $\psi\in I_{+}^W$.
We observe that if $\nu$ equals zero on
all functions from $I_{+}^W$
(i.e., assigns zero integrals to such functions), then the same is true for the solution $\mu_t$ to the  Cauchy problem.

Set
$$
\mathcal{M}_{\alpha}(V)=
\Bigl\{\mu\in\mathcal{P}_V(\mathbb{R}^d)\colon
\int_{\mathbb{R}^d} V\,d\mu\le \alpha\Bigr\}.
$$
Recall that $W=V^{\gamma}$, where $\gamma\in(0, 1/2]$, and $\|\mu\|_{W}
=\|W\mu\|_{TV}$.

Suppose that for every $\varepsilon\in [0, 1)$ we are given a mapping
$$
b_{\varepsilon}(\,\cdot \,, \,\cdot\,)\colon \mathbb{R}^d\times\mathcal{P}_V(\mathbb{R}^d)\to\mathbb{R}^d
$$
such that for every $\mu\in\mathcal{P}_V(\mathbb{R}^d)$
the mapping $x\mapsto b_{\varepsilon}(x, \mu)$ is Borel.
Let
$$
L_{\mu, \varepsilon}u(x)={\rm trace}\, (A(x)D_x^2u(x))+\langle b_{\varepsilon}(x,\mu),\nabla u(x)\rangle.
$$
Suppose that for every measure $\nu\in\mathcal{P}_V(\mathbb{R}^d)$ with
$\nu|_{I_{+}^W}=0$ there exist numbers $C>0$,
$\Lambda>0$ and $\delta\in[0, 1]$ and a positive
function $N_1$ on $[0, +\infty)$ (thus for different $\nu$ these objects can be different)
such that

\, ${\rm{\bf (H_1)}}$ \, for all $\varepsilon\in[0, 1)$, $\alpha\ge 1$ and
$\mu\in\mathcal{M}_{\alpha}(V)$ satisfying the conditions
$\mu|_{I_{+}^W}=0$ and $\mu|_{I_0^W}=\nu|_{I_0^W}$, we have
$$
L_{\mu, \varepsilon}V(x)\le (1-\delta)C+\Lambda(\delta\alpha-V(x))
\quad
\forall x\in\mathbb{R}^d,
$$

\, ${\rm {\bf(H_2)}}$ \, for all $\varepsilon\in[0, 1)$,  $\alpha$ and
$\mu\in\mathcal{M}_{\alpha}(V)$
satisfying the conditions $\mu|_{I_{+}^W}=0$ and $\mu|_{I_0^W}=\nu|_{I_0^W}$,
we have
$$
|b_{\varepsilon}(x, \mu)|\le N_1(\alpha)V(x)^{\frac{1}{2}-\gamma} \quad
\forall x\in\mathbb{R}^d.
$$

Suppose that there exists a positive function $N_2$
on $[0, +\infty)$ such that

\, ${\rm{\bf (H_3)}}$ \, for all $\varepsilon\in[0, 1)$,  $\alpha\ge0$ and $\mu,
\sigma\in\mathcal{M}_{\alpha}(V)$ satisfying the conditions
$\mu|_{I_{+}^W}=\sigma|_{I_{+}^W}=0$ and $\mu|_{I_0^W}=\sigma|_{I_0^W}$,
 we have
$$
|b_{\varepsilon}(x, \mu)-b_{\varepsilon}(x, \sigma)|\le \varepsilon
N_2(\alpha)V(x)^{\frac{1}{2}-\gamma}\|\mu-\sigma\|_{W} \quad
\forall x\in\mathbb{R}^d.
$$
Note that if
$$
b_{\varepsilon}(x, \mu)=\int_{\mathbb{R}^d} K_{\varepsilon}(x, y)\,\mu(dy) +\widetilde{b_{\varepsilon}}(x, \mu)
$$
and for every $x$ the
function $y\mapsto K_{\varepsilon}(x, y)$ belongs to $I_0^W$, then condition ${\rm (H_3)}$
refers only to $\widetilde{b_{\varepsilon}}$, since the difference
of the integrals of $K_{\varepsilon}(x, y)$ with respect to two measures $\mu$ and $\sigma$
with $\mu|_{I_0^W}=\sigma|_{I_0^W}$ is zero.

For example, this is the case where $d=1$, $A=I$ and
$$
b_{\varepsilon}(x, \mu)=-x+\int_{\mathbb{R}} y\,\mu(dy).
$$
Here $b_{\varepsilon}$ does not depend on $\varepsilon$.
The function $x\mapsto x$ belongs to $I_0^W$ and for every measure $\mu$ satisfying the equality
$$
\int_{\mathbb{R}} x\,\mu(dx)=\int_{\mathbb{R}} x\,\nu(dx)=Q
$$
we have
$$
L_{\mu, \varepsilon}(1+|x|^2)\le 3+Q^2-(1+|x|^2), \quad |b(x, \mu)|\le |Q|+|x|\le (1+|Q|)(1+|x|^2)^{1/2},
$$
i.e., conditions ${\rm (H_1)}$, ${\rm (H_2)}$ and ${\rm (H_3)}$ are fulfilled with
$$C=2+Q^2, \quad \delta=0, \quad \Lambda=1,\quad \gamma=1/2, \quad N_1=1+|Q|, \quad N_2=0.$$

Note that ${\rm (H_3)}$ obviously holds at $\varepsilon=0$ if $b_0$ does not depend on~$\mu$, i.e.,
the equation at $\varepsilon=0$ becomes linear, but the previous example shows that this condition can hold
also in case of a nontrivial dependence on~$\mu$.

The main result of this paper (presented in the next section)
states that, for all sufficiently
small $\varepsilon$, the listed conditions ensure the exponential convergence
to the stationary distribution.

\section{Convergence to stationary solutions}

Suppose that for every $\nu\in\mathcal{P}_V(\mathbb{R}^d)$
 there is a solution  $\{\mu_t\}$
to the   problem~(\ref{eq1A}) on~$[0, +\infty)$. Sufficient conditions for the existence of solutions
to parabolic and stationary Fokker--Planck--Kolmogorov equations are discussed
in the last section.
It is immediate  that $\mu_t|_{I_0^W}=\nu|_{I_0^W}$ according to Proposition \ref{pr00}.
Moreover, if $\nu|_{I_{+}^W}=0$, then $\mu_t|_{I_{+}^W}=0$
according to Proposition~\ref{pr02}.
We shall now use conditions
${\rm (H_1)}$ -- ${\rm (H_3)}$ introduced
at the end of the previous section.

\begin{theorem}\label{th1}
Suppose that conditions ${\rm (H_1)}$, ${\rm (H_2)}$ and ${\rm (H_3)}$ are fulfilled.
Let $\nu\in\mathcal{P}_V(\mathbb{R}^d)$, $\nu|_{I_{+}^W}=0$ and $\alpha>0$.
Then there exist positive numbers $\varepsilon_0$, $\alpha_1$ and $\alpha_2$
{\rm(}depending on $\nu$ and $\alpha${\rm)}
such that, whenever $\varepsilon\in [0, \varepsilon_0)$,
for the solution $\mu_t$
to the Cauchy problem {\rm (\ref{eq1A})} with coefficients $A$ and $b_{\varepsilon}$
and initial data $\nu$ and the stationary solution $\mu$ to  equation {\rm (\ref{eq1AS})}
with  coefficients $A$ and $b_{\varepsilon}$ such that
$$
\mu|_{I_0^W}=\nu|_{I_0^W} \quad \hbox{and} \quad \int_{\mathbb{R}^d} V\,d\mu\le \alpha,
$$
we have
$$
\|\mu_t-\mu\|_W\le \alpha_1e^{-\alpha_2t} \quad \forall\, t\ge 0.
$$
\end{theorem}

\begin{example}\rm
Let $d\ge 1$, $A=I$,
$$
b_{\varepsilon}(x, \mu)=-Rx+\int_{\mathbb{R}^d}\langle v, y\rangle\,\mu(dy)h+\varepsilon\int_{\mathbb{R}^d} H(x, y)\,\mu(dy),
$$
where $R$ is a constant matrix, $v$ and $h$ are constant vectors and
$$R^{*}v=\lambda v, \quad \langle v, h\rangle=\lambda, \quad \langle H(x, y), v\rangle=0.$$
Suppose also that
$$
\langle Rx, x\rangle \ge q|x|^2, \quad q>0, \quad \sup_{x, y}|H(x, y)|<\infty.
$$
For example, for $d=2$ one can take $R=2I$, $v=(1, 1)$ and $h=(1, 1)$:
$$
b^1_{\varepsilon}(x, \mu)=-2x_1+\int_{\mathbb{R}^2}(y_1+y_2)\,\mu(dy)+\varepsilon\int_{\mathbb{R}^2} H(x, y)\,\mu(dy),
$$
$$
b^2_{\varepsilon}(x, \mu)=-2x_2+\int_{\mathbb{R}^2}(y_1+y_2)\,\mu(dy)-\varepsilon\int_{\mathbb{R}^2} H(x, y)\,\mu(dy).
$$
Let us show that all conditions of Theorem \ref{th1} are fulfilled.
The function $x\mapsto \langle v, x\rangle$ belongs to~$I_0^W$.
Let $\nu$ be  a probability measure with $|x|^2\in L^1(\nu)$.
Set
$$
Q=\int_{\mathbb{R}^2} \langle v, y\rangle\,\nu(dy).
$$
For all measures $\sigma$ that coincide with $\nu$ on $I_0^W$ we have
$$
L_{\sigma,\varepsilon}(1+|x|^2)\le 2d+q+q^{-1}(|h||Q|+\sup_{x, y}|H(x, y)|)^{2}-q(1+|x|^2),
$$
$$
|b_{\varepsilon}(x, \mu)|\le (\|R\|+|h||Q|+\sup_{x, y}|H(x, y)|)(1+|x|^2)^{1/2}.
$$
Finally, for every two measures $\mu$ and $\sigma$ that coincide
on $I_0^W$ we have
$$
|b_{\varepsilon}(x, \mu)-b_{\varepsilon}(x, \sigma)|\le \varepsilon
\sup_{x, y}|H(x,y)|\|(\mu-\sigma)(1+|x|^2)^{1/2}\|_{TV}.$$
Thus, conditions ${\rm (H_1)}$, ${\rm (H_2)}$ and ${\rm (H_3)}$ are fulfilled
with $\gamma=1/2$, $W(x)=(1+|x|^2)^{1/2}$ $V(x)=1+|x|^2$, $\delta=0$, $\Lambda=1$,
$N_2=\sup_{x, y}|H(x,y)|$ and
$$
C=2d+q+q^{-1}(|h||Q|+\sup_{x, y}|H(x, y)|)^{2},
\quad N_1=\|R\|+|h||Q|+\sup_{x, y}|H(x, y)|.
$$
Moreover, it will be shown in Section~4 (see Example \ref{ex-st})
that for every $\varepsilon\in(0, 1)$ and every number $Q$ there is a stationary solution $\mu$
such that
$$
Q=\int_{\mathbb{R}^2} \langle v, y\rangle\,\mu(dy).
$$
In addition, for this solution $\mu$ we have
$$
\int_{\mathbb{R}^2} (1+|x|^2)\,\mu(dx)\le 2dq^{-1}+1+q^{-2}(|h||Q|+\sup_{x, y}|H(x, y)|)^{2}.
$$
Thus, for every number $Q_0>0$ there is  a number $\varepsilon_0>0$, depending
only on $Q_0$, such that, for any $\varepsilon\in[0, \varepsilon_0)$,
 $Q\in (-Q_0, Q_0)$ and a probability measure $\nu$ such that the integral of $\langle v, y\rangle$
with respect to $\nu$ equals $Q$ and $|x|^2\in L^1(\nu)$,
the solutions $\mu_t$ to the Cauchy problem (\ref{eq1A}) with initial data $\nu$
converge to the stationary solution $\mu$ and
$$
\|(1+|x|)(\mu_t-\mu)\|_{TV}\le \alpha_1e^{-\alpha_2t},
$$
where $\alpha_1$ and $\alpha_2$ depend only on $Q_0$ and $\|(1+|x|)^2\|_{L^1(\nu)}$.
\end{example}

Theorem \ref{th1} is of a somewhat conditional nature: 1) we assume
that there  exists a stationary solution $\mu$ such that $\nu|_{I_0^W}=\mu|_{I_0^W}$,
2) $\varepsilon_0$ depends on $\nu$ and, what is worse,  also on $\mu$.
Dependence on $\mu$ arises in connection
with dependence of our conditions on $\alpha$,
i.e., due to nonlinearity (see the discussion in the proof of Lemma \ref{lem1}
and before Lemma \ref{lem2}).
If we require stronger restrictions on the coefficients,
then we can avoid such dependence.

\begin{theorem}\label{th2}
Suppose that in place of conditions ${\rm (H_1)}$, ${\rm (H_2)}$
and ${\rm (H_3)}$ there exist positive numbers $C_1$, $C_2$
and positive functions $N_1$ and $N_2$
such that for all $\varepsilon\in[0, 1)$,  $\alpha>0$ and $\mu, \sigma\in\mathcal{M}_{\alpha}(V)$
we have $L_{\mu,\varepsilon}V\le C_1-C_2V$ and
$$
|b_{\varepsilon}(x, \mu)|\le N_1(\alpha)V^{1/2-\gamma}(x), \quad
|b_{\varepsilon}(x, \mu)-b_{\varepsilon}(x, \sigma)|\le \varepsilon N_2(\alpha)V(x)^{1/2-\gamma}
\|\mu-\sigma\|_{W}.
$$
Then there exists  $\varepsilon_0>0$,
such that, for each $\varepsilon\in[0, \varepsilon_0)$
there exists a stationary solution $\mu$ and,
for every measure $\nu\in\mathcal{P}_V(\mathbb{R}^d)$,
for the solution $\{\mu_t\}$ to the Cauchy problem
{\rm (\ref{eq1A})} with the initial condition $\nu$
 one has
$$
\|\mu_t-\mu\|_W\le \alpha_1e^{-\alpha_2t} \quad \forall\, t\ge 0,
$$
where $\alpha_1, \alpha_2$ are positive numbers
such that $\alpha_2$ does not depend on $\nu$.
\end{theorem}

\begin{example}\label{ex-main1}\rm
Suppose that  $A=I$ and there exist numbers $m\ge 1$, $\gamma_1>0$, $\gamma_2>0$
and positive functions $N_1$, $N_2$ such that
$$
\langle b_{\varepsilon}(x, \mu), x\rangle\le \gamma_1-\gamma_2|x|^2, \quad
|b_{\varepsilon}(x, \mu)|\le N_1(\alpha)(1+|x|)^m,
$$
$$
|b_{\varepsilon}(x, \mu)-b_{\varepsilon}(x, \sigma)|\le \varepsilon N_2(\alpha)(1+|x|)^m\|(1+|y|)^m
(\mu-\sigma)\|_{TV}
$$
for all $\varepsilon\in[0, 1)$, $\alpha>0$ and $\mu, \sigma\in\mathcal{M}_{\alpha}((1+|x|^{2m+1})$.
Hence all conditions of Theorem \ref{th2} are fulfilled with $V(x)=(1+|x|^2)^{m+1/2})$ and $W(x)=(1+|x|^2)^{m/2}$.
In particular, the listed conditions are fulfilled if
$$
b_{\varepsilon}(x, \mu)=b_0(x)+\varepsilon \int_{\mathbb{R}^d} K(x, y)\,\mu(dy),
$$
where
$$
\langle b_0(x), x\rangle\le c_1-c_2|x|^2, \quad \langle K(x, y), x\rangle \le c_3+c_3|x|^2, \quad c_3<c_2,
$$
$$
|b_0(x)|\le c_4+c_4|x|^m, \quad |K(x, y)|\le c_5(1+|x|^m)(1+|y|^m)
$$
with some positive numbers $c_1, c_2, c_3, c_4$ and $c_5$.
\end{example}

The existence of a  stationary solution $\mu$ under the conditions of Theorem \ref{th2}
will be established in the next section in  Proposition \ref{pr1}. As it will be explained  in
Remark \ref{rem1}, under the conditions of Theorem \ref{th2}
there exists a stationary  solution $\mu$ with
$$
\|V\|_{L^1(\mu)}\le C_1/C_2.
$$
It is the stationary solution that we need in the proof of Theorem \ref{th2}.

We only give the proof of  Theorem \ref{th1}, since the  proof of Theorem \ref{th2}
differs by minor technical details that will be discussed in the course of the proof.

Below for shortening notation and reducing the number of indices we omit the index
$\varepsilon$ and in place of
$b_{\varepsilon}(x, \mu)$ and $L_{\mu, \varepsilon}$ we write $b(x, \mu)$ and $L_{\mu}$.

The plan of the proof is this: 1) we verify  that convergence holds
for solutions $\eta_t$ to the linear equation with the coefficient  $b(x, \mu)$,
in which we substitute the stationary solution~$\mu$,
2)~we obtain an estimate on the distance  $\|\eta_t-\mu_t\|_W$, 3) we prove
that for some $T>0$
one has a contraction $\|\mu_T-\mu\|_W\le q\|\nu-\mu\|_W$ with $q<1$.

\begin{lemma}\label{lem1}
Suppose that we are in the situation of Theorem~{\rm \ref{th1}} or~{\rm\ref{th2}}
with the corresponding $\mu$ and~$\nu$.
Then there exist numbers $N>0$ and $\lambda>0$ such that for every $t>0$ we have
$$
\|\eta_t-\mu\|_W\le Ne^{-\lambda t}\|\mu_0-\mu\|_W,
$$
where $\mu_0\in \mathcal{P}_V(\mathbb{R}^d)$ and $\{\eta_t\}$ is the
solution to the Cauchy problem
$$
\partial_t\eta_t=L_{\mu}^{*}\eta_t, \quad \eta_0=\mu_0.
$$
The numbers $N$ and $\lambda$ depend on $\|V\|_{L^1(\mu)}$ and $\nu$,
and if the condition of Theorem~{\rm\ref{th2}} is fulfilled,
then $N$ and $\lambda$ depend on $C_1$ and $C_2$, but not on $\mu$ and $\nu$.
\end{lemma}
\begin{proof}
Let $\{T_t\}_{t\ge0}$ be the Markov semigroup on $L^1(\mu)$  with  generator
$$
L\varphi(x)={\rm trace}(A(x)D^2\varphi(x))+\langle b(x, \mu), \nabla\varphi(x)\rangle
$$
on $C_0^\infty(\mathbb{R}^d)$. This semigroup exists and is unique
under our assumptions, see \cite[Theorem~5.2.2, Proposition~5.2.5 and Example~5.5.1]{BKRSH-book}.
Moreover, $\eta_t=T_t^{*}\mu_0$. By  \cite[Theorem 6.4.7]{BKRSH-book}
there exists a positive continuous function $\varrho(x, y, t)$
such that
$$
T_tf(x)=\int_{\mathbb{R}^d}\varrho(x, y, t)f(y)\,dy.
$$
Moreover, $\varrho(x, y, t)$  satisfies
the Cauchy problem $\partial_t\varrho=L^{*}\varrho$ with respect to $(t, y)$
with the initial condition $\delta_x$.
Set $\alpha=\|V\|_{L^1(\mu)}$ (or $\alpha=C_1/C_2$ in the case of
 Theorem~\ref{th2})
and $\Lambda_1=(1-\delta)C+\Lambda\delta\alpha$ (or $\Lambda_1=C_1$
and $\Lambda=C_2$ in the case of Theorem~\ref{th2}).
Recall that $LV\le \Lambda_1-\Lambda V$.
Since
$$LW=\gamma V^{\gamma-1}LV
+\gamma(\gamma-1)V^{\gamma-2}|\nabla V|^2\le \gamma\Lambda_1-\gamma\Lambda W,
$$
we have
$$
\partial_t(We^{tH_3})+L(We^{tH_3})\le \gamma\Lambda_1e^{t\Lambda}.
$$
By \cite[Theorem 7.1.1]{BKRSH-book}
$$
\int_{\mathbb{R}^d}W(y)\varrho(x, y, t)\,dy\le e^{-\Lambda t}W(x)+\gamma
\Lambda_1\Lambda^{-1}(1-e^{-\Lambda t}).
$$
Let us fix a number $\tau>0$ such that $\gamma \Lambda_1\Lambda^{-1}
(1-e^{-\Lambda \tau})<1$.
Note that $\tau$ depends on $C_1$ and $C_2$ in the case of Theorem~\ref{th2}.
Then
$$
T_{\tau}W(x)\le e^{-\Lambda\tau}W(x)+1.
$$
The function $Q(r)=\max_{|x|\le 2r}V(x)^{\frac{1}{2}-\gamma}$ is continuous
and increasing on $[0, +\infty)$.
The condition $|b(x, \mu)|\le N_1(\alpha)V(x)^{\frac{1}{2}-\gamma}\le
N_1(\alpha)Q(|x|/2)$
and Harnack's inequality (see \cite[Theorem 8.2.1]{BKRSH-book})
imply that, for every $x\in B(0, R)$ and $y\in\mathbb{R}^d$, we have
$$
\varrho(x, y, \tau)\ge \varrho(x, 0, \tau/2)e^{-K(\tau)(1+Q^2(|y|)+|y|^2)}
\ge m_1(R)e^{-K(\tau)(1+Q^2(|y|)+|y|^2)},
$$
where $m_1(R)=\min_{x\in B(0, R)}\varrho(x, 0, \tau/2)$.
The number $K(\tau)$ depends only on the matrix $A$, $\tau$ and the dimension~$d$,
and there is an explicit expression for $K(\tau)$ in \cite[Theorem~8.2.1]{BKRSH-book}.
Note that so far $m_1$ depends in a very complicated way on the stationary  measure $\mu$,
since $L$ depends on $\mu$ and $\varrho$ defines the operator $L$. We would like
to have dependence  only on $N_1$,
$\Lambda_1$ and $\Lambda$, which in turn depend on $\alpha=\|V\|_{L^1(\mu)}$
and $\nu$, and in the case of Theorem \ref{th2} depend
on $C_1$ and $C_2$ and are independent of $\nu$ and $\mu$.
Thus, we have to estimate $m_1$ from below. Let $\psi\in C_0^{\infty}(\mathbb{R}^d)$,
$\psi(x)=1$ if $y\in B(0, 2R)$ and $\psi(y)=0$ if $y\notin B(0, 3R)$.
Let $x\in B(0, R)$. Then
$$
\int_{\mathbb{R}^d}\psi(y)\varrho(x, y, t)\,dy=\psi(x)
+\int_0^t\int_{\mathbb{R}^d}L\psi(y)\varrho(x,y,s)\,dy\,ds.
$$
Therefore,
$$
\int_{\mathbb{R}^d}\psi(y)\varrho(x, y, t)\,dy\ge 1-t\sup_{y}|L\psi(y)|.
$$
Choosing  $t$ so small  that the right-hand side is estimated from below by $1/2$, we obtain
$$
\sup_{y\in B(0, 3R)}\varrho(x, y, t)\ge 1/2
\quad \forall\, x\in B(0,R).
$$
Decreasing  $\tau$ if necessary, we can assume  that $t=\tau/4$.
Applying again Harnack's inequality from \cite[Theorem 8.1.3]{BKRSH-book}
we obtain the estimate
$$
1/2\le C\varrho(x, 0, \tau/2),
$$
where $C$ depends only on $R$, $Q$, and $\tau$. Thus,
$$
\varrho(x, y, \tau)\ge m(R)e^{-K(\tau)(1+Q^2(|y|)+|y|^2)}
\quad \forall\, x\in B(0,R)
$$
where $m(R)$ and $\tau$ depend only on $N_1$, $\Lambda_1$ and $\Lambda$.

Let us  now recall the Harris ergodic theorem (see \cite{HM}).
Let $\mathcal{P}(\,\cdot\, , \, \cdot \,)$ be a Markov transition
kernel  defined on a measurable space~$(X,\mathcal{B})$, i.e.,
for each $x\in X$, the function $B\mapsto \mathcal{P}(x,B)$ is a
probability measure on~$\mathcal{B}$,
and, for each $B\in \mathcal{B}$, the function $x\mapsto
\mathcal{P}(x,B)$ is $\mathcal{B}$-measurable.
The transition kernel defines operators on functions and
measures by setting
$$
\mathcal{P} f(x)=\int_X f(y)\, \mathcal{P}(x,dy),
$$
$$
\mathcal{P}\sigma(B)=\int_X \mathcal{P}(x,B)\, \sigma(dx).
$$

Let us assume that

\, (i) \, there exist a function $U\colon X\to [0,+\infty)$ and numbers  $\delta\in(0, 1)$ and $K$
such that
$$
\mathcal{P}U(x)\le \delta U(x)+K
\quad \forall\, x\in X;
$$

\, (ii) \, there exist a number $q\in(0, 1)$ and a probability
measure $\sigma$ such that
$$
\inf_{x\colon U(x)\le R}\mathcal{P}(x, \,\cdot\,)\ge q\sigma(\,\cdot\,),
$$
for some $R>2K/(1-\delta)$.

According to  \cite[Theorem 1.3]{HM}, there exist numbers
 $\beta_0\in(0, 1)$ and $\beta>0$ such that
$$
\|\mathcal{P}\mu_1-\mathcal{P}\mu_2\|_{1+\beta U}\le
\beta_0\|\mu_1-\mu_2\|_{1+\beta U}
$$
for every pair of probability measures $\mu_1$ and $\mu_2$ on $X$.
From this estimate one can derive the bound
$$
\|\mathcal{P}^n\nu-\mu\|_{1+\beta U}\le \beta_0^n\|\nu-\mu\|_{1+\beta U}
$$
for the stationary measure $\mu$ (that is, $\mathcal{P}\mu=\mu$) and
every measure $\nu$.

We can now apply this assertion to the Markov transition
kernel $\varrho(x, y, \tau)\,dy$
with $U(x)=W(x)$, $K=1$, $\delta=e^{-\Lambda\tau}$ and
$$
q\sigma(dy)=m(R)e^{-K(\tau)(1+Q^2(|y|)+|y|^2)}\, dy,
$$
 where the number $R$ is larger than $2/(1-e^{-\Lambda\tau})$.
Therefore, we have
\begin{equation}\label{est-st}
\|T_{n\tau}^{*}\mu_0-\mu\|_{W}\le N_1\beta_0^n\|\nu-\mu\|_{W}.
\end{equation}
Note that for every $\varphi\in C_0^{\infty}(\mathbb{R}^d)$ such
that $|\varphi(x)|\le W(x)$ for all~$x$ we have
$$
|T_t\varphi|\le T_t|\varphi|\le T_tW\le 2W \quad \forall t\in(0, \tau).
$$
Hence
$$
\int_{\mathbb{R}^d}\varphi\,d(T_t^{*}\mu_0 -\mu)
=\int_{\mathbb{R}^d}T_t\varphi\,d(\mu_0-\mu)\le
2\|\mu_0-\mu\|_{W}
$$
and we obtain the estimate
$$
\sup_{t\in(0, \tau)}\|T_t^{*}\mu_0 -\mu\|_{W}\le 2\|\mu_0 -\mu\|_{W}.
$$
Summing this estimate and (\ref{est-st}), we complete the proof.
\end{proof}

Suppose that as above $\mu_t$ is the solution to the Cauchy problem  (\ref{eq1A}) with initial data~$\nu$.
Before  estimating the distance  between $\mu_t$ and $\eta_t$ we estimate
$\|V\|_{L^1(\mu_t)}$.
Now let $\delta<1$. With the aid of condition ${\rm (H_1)}$ we deduce that
$$
\int_{\mathbb{R}^d}V\,d\mu_t\le \int_{\mathbb{R}^d}V\,d\nu+(1-\delta)Ct
-(1-\delta)\int_0^t\int_{\mathbb{R}^d}V\,d\mu_s\,ds.
$$
If $\delta=1$, then $\|V\|_{L^1(\mu_t)}\le \|V\|_{L^1(\nu)}$.
If $\delta<1$, then by Gronwall's inequality we obtain
$$
\int_{\mathbb{R}^d}V\,d\mu_t\le \Bigl(\int_{\mathbb{R}^d}V\,d\nu
-\frac{C}{\Lambda}\Bigr)e^{-\Lambda(1-\delta)t}+\frac{C}{\Lambda}.
$$
If $\delta<1$, then, starting from some $\tau_0>0$, we can assume  that
$\|V\|_{L^1(\mu_t)}\le C\Lambda^{-1}+1$ for all $t\ge \tau_0$. Since we are interested in convergence
as $t\to\infty$,  we can always consider the Cauchy problem for $t>\tau_0$ and
with the initial condition $\mu_{\tau_0}$ in place of $\nu$. Hence we assume further that
for $\delta=1$ we have  $\|V\|_{L^1(\mu_t)}\le\|V\|_{L^1(\nu)}$,
and for $\delta<1$ we have $\|V\|_{L^1(\mu_t)}\le C\Lambda^{-1}+1$.
Let
$$
\theta=\max\{\|V\|_{L^1(\nu)}, C\Lambda^{-1}+1, \|V\|_{\mu}\},
$$
and let $\theta=1+C_1/C_2$ in the conditions of Theorem \ref{th2} .
We observe that by the uniqueness of solutions to the Cauchy problem
(\ref{eq1A}) under our assumptions about the coefficients (see Proposition \ref{pr2})
the solution $\mu_{\tau+t}$ to the Cauchy problem with the initial condition $\nu$ coincides for $t\ge 0$
with the solution $\mu_t$ to the Cauchy with the initial condition~$\mu_{\tau}$.

\begin{lemma}\label{lem2}
Let $\tau\ge \tau_0$.
Let $\{\mu_t\}$ be the solution to the Cauchy problem {\rm (\ref{eq1A})}
and let $\{\eta_t\}$ be the
solution to the Cauchy problem $\partial_t\eta_t=L_{\mu}^{*}\eta_t$,
$\eta_0=\mu_{\tau}$.
Then
$$
\|\mu_{\tau+t}-\eta_t\|_W\le \varepsilon C(\theta)
\biggl(\int_0^t\|\mu_{\tau+t}-\mu\|_{W}^2\,dt\biggr)^{1/2}, \quad C(\theta)=\theta\bigl(8K_1^2N_2(\theta)
\theta^2t+\theta\bigr).
$$
\end{lemma}
\begin{proof}
Let $\widetilde{\mu_t}=\mu_{\tau+t}$.
Let $\{T_t\}_{t\ge0}$ be the semigroup with the generator $L$ from
the previous proof. Set
$$
u(x, t)=T_{s-t}\psi(x),
$$
where
$\psi\in C_0^{\infty}(\mathbb{R}^d)$ and $|\psi(x)|\le W(x)$ for all~$x$.
Then
$$
\int_{\mathbb{R}^d}\psi\,d(\widetilde{\mu_t}-\eta_t)=
\int_0^{t}\int_{\mathbb{R}^d} \langle b(x, \widetilde{\mu_t})-b(x, \mu),
\nabla_x u\rangle\,\widetilde{\mu_t}(dx)\,dt.
$$
We need a bound on $|\nabla_x u|$. We have
\begin{multline*}
\int_{\mathbb{R}^d} \psi(x)^2\,\widetilde{\mu_t}(dx)-\int_{\mathbb{R}^d}
u(x, 0)^2\,\mu_{\tau}(dx)
\\
=
\int_0^{t}\int_{\mathbb{R}^d} \Bigl[2|\sqrt{A}\nabla u(x)|^2
+2\langle b(x, \widetilde{\mu_t})-b(x, \mu), \nabla_x u(x)\rangle u(x)
\Bigr]\,\widetilde{\mu_t}(dx)\,dt.
\end{multline*}
Recall that $|u|\le 2W$ and
$$
|b(x, \widetilde{\mu_t})-b(x, \mu)|\le
\varepsilon N_2(\theta)V(x)^{\frac{1}{2}-\gamma}\|\widetilde{\mu_t}-\mu\|_W,
$$
where $\|\widetilde{\mu_t}-\mu\|_W\le 2\theta^{1/2}$.
Since
$$
2|\langle b(x, \widetilde{\mu_t})-b(x, \mu), \nabla_x u(x)\rangle u(x)|\le
4K_1N_2(\theta)^2\theta V(x)+2^{-1}K_1^{-1}|\nabla_x u|^2,
$$
we obtain
$$
\int_0^{t}\int_{\mathbb{R}^d} |\nabla u|^2\,d\widetilde{\mu_t}\,dt\le
8K_1^2N_2(\theta)\theta^2t+\theta.
$$
Let us observe that
\begin{multline*}
\int_0^{t}\int_{\mathbb{R}^d} \langle b(x, \widetilde{\mu_t})-b(x, \mu),
\nabla_x u(x)\rangle\,\widetilde{\mu_t}(dx)\,dt
\\
\le
\varepsilon\theta\Bigl(8K_1^2N_2(\theta)\theta^2t+\theta\Bigr)
\biggl(\int_0^t\|\widetilde{\mu_t}-\mu\|_W^2\,dt\biggr)^{1/2}.
\end{multline*}
We obtain
$$
\int_{\mathbb{R}^d}\psi\,d(\widetilde{\mu_t}-\eta_t)\le
\varepsilon\theta\Bigl(8K_1^2N_2(\theta)\theta^2t+\theta\Bigr)
\biggl(\int_0^t\|\widetilde{\mu_t}-\mu\|_W^2\,dt\biggr)^{1/2}.
$$
Since $|\psi|\le W$,
we have
$$
\|\widetilde{\mu_t}-\eta_t\|_W\le \varepsilon\theta\Bigl(8K_1^2N_2(\theta)\theta^2t+\theta\Bigr)
\biggl(\int_0^t\|\widetilde{\mu_t}-\mu\|_W^2\,dt\biggr)^{1/2}.
$$
which completes the proof.
\end{proof}

We are now ready to prove our main theorems.

\begin{proof}[Proof of Theorem~\ref{th1}]
Let $\tau\ge \tau_0$. We recall that $\mu_{\tau+t}$ with $t>0$ solves the Cauchy problem (\ref{eq1A})
with the initial condition $\mu_{\tau}$. Let $\eta_t$ be the solution to the
linear Cauchy problem $\partial_t\eta_t=L_{\mu}^{*}\eta_t$,
$\eta_0=\mu_{\tau}$.
Using Lemma \ref{lem1} and Lemma \ref{lem2} we obtain
\begin{align*}
\|\mu_{\tau+t}-\mu\|_W
&\le \|\eta_t-\mu\|_W+\|\mu_{\tau+t}-\eta_t\|_W
\\
&\le
Ne^{-\lambda t}\|\mu_{\tau}-\mu\|_W+
\varepsilon\theta\Bigl(8K_1^2N_2(\theta)\theta^2t+\theta\Bigr)
\biggl(\int_0^t\|\mu_{\tau+t}-\mu\|_W^2\,dt\biggr)^{1/2}.
\end{align*}
Let $T>0$ be such that $Ne^{-\lambda T}<1/2$. Set
$$
M=\theta\Bigl(8K_1^2N_2(\theta)\theta^2T+\theta\Bigr).
$$
For any $t\in[0, T]$ we have
$$
\|\mu_{\tau+t}-\mu\|_W\le N\|\mu_{\tau}-\mu\|_W+\varepsilon M
\biggl(\int_0^t\|\mu_{\tau+t}-\mu\|_W^2\,dt\biggr)^{1/2}.
$$
By Gronwall's inequality
\begin{equation}\label{est-st2}
\|\mu_{\tau+t}-\mu\|_W\le 2N e^{2\varepsilon^2M^2t}\|\mu_{\tau}-\mu\|_W\le
2N e^{2\varepsilon^2M^2T}\|\mu_{\tau}-\mu\|_W.
\end{equation}
Therefore,
$$
\|\mu_{\tau+T}-\mu\|_W\le \Bigl(\frac{1}{2}+2N\varepsilon MT^{1/2}
e^{2\varepsilon^2M^2T^2}\Bigl)\|\mu_{\tau}-\mu\|_W.
$$
Let us take $\varepsilon_0>0$ such that
$$
q=\frac{1}{2}+2N \varepsilon MT^{1/2}e^{2\varepsilon^2M^2T^2}<1
\quad \forall \varepsilon\in(0, \varepsilon_0).
$$
Then
$$
\|\mu_{\tau+T}-\mu\|_W\le q\|\mu_{\tau}-\mu\|_W \quad \forall \tau\ge\tau_0.
$$
We have
$\|\mu_{\tau_0+nT}-\mu\|_W\le \theta^n\|\mu_{\tau_0}-\mu\|_W$.
Using this estimate and (\ref{est-st2}), we obtain
$$
\|\mu_t-\mu\|_W\le \alpha_1e^{-\alpha_2 t} \quad \forall t\ge 0,
$$
which completes the proof.
\end{proof}

It is seen from the  proof above  that $\varepsilon_0$ depends on the quantities $N$ and
$\lambda$ from Lemma \ref{lem1}
and on the number $\theta$ defined before Lemma~\ref{lem2}. Thus,
under the conditions of Theorem~\ref{th2} the number
$\varepsilon_0$ depends only on $C_1$, $C_2$ and the functions $N_1$ and~$N_2$.

\begin{remark}\label{r3.7}
\rm
Let us discuss in more detail Example \ref{ex1.1} from the introduction.
Let $d=1$, $A=1$ and
$$
b_{\varepsilon}(x, \mu)=-x+\varepsilon B(\mu), \quad B(\mu)=\int x\,\mu(dx), \quad \varepsilon\ge 0.
$$
In the case where $0\le\varepsilon<1$, the standard Gaussian measure $\mu$ is the unique probability
solution to the stationary equation. Let us show that for every initial
condition $\nu$ with a finite first moment the measures $\mu_t$ from the solution
  to the Cauchy problem converge to~$\mu$.
The justification repeats the main steps of the proof of Theorem~\ref{th1}.
We observe that
$$
\frac{d}{dt}B(\mu_t)=-(1-\varepsilon)B(\mu_t), \quad B(\mu_{t+s})=e^{-(1-\varepsilon)t}B(\mu_{s}).
$$
Moreover,
$$
\int|x|^2\,d\mu_t\le \int|x|^2\,\nu(dx) +(1-\varepsilon)^{-1}\int|x|\, \nu(dx).
$$
Let $\tau>0$ and let $\{\eta_t\}$ be the solution to the Cauchy problem
$$
\partial_t\eta_t=\eta_t''+(x\eta_t)', \quad \eta_0=\mu_{\tau}.
$$
By Lemma \ref{lem1} with $W(x)=(1+|x|^2)^{1/2}$ and $V(x)=1+|x|^2$ we have
$$
\|\eta_t-\mu\|_{W}\le C_1e^{-C_2\tau}\|\mu_{\tau}-\mu\|_{W}\le C_3e^{-C_2\tau}.
$$
Since
$$b(x, \mu_{t+\tau})-b(x, \mu)=\varepsilon B(\mu_{t+\tau})=\varepsilon e^{-(1-\varepsilon)t}B(\mu_{\tau}),$$
repeating the reasoning from Lemma \ref{lem2} one can readily obtain
the bound
$$
\|\mu_{t+\tau}-\eta_t\|_{W}\le C(\varepsilon)B(\mu_{\tau})\le C_4e^{-(1-\varepsilon)\tau}, \quad t\in[0, \tau].
$$
Combining the obtained estimates for $t=\tau$, we conclude that
$$
\|\mu_{2\tau}-\mu\|_{W}\le \|\mu_{2\tau}-\eta_{\tau}\|_{W}+\|\eta_{\tau}-\mu\|_{W}\le
C_4e^{-(1-\varepsilon)\tau}+C_3e^{-C_2\tau}.
$$
Thus,
$$
\|\mu_{t}-\mu\|_{W}\le \alpha_1e^{-\alpha_2 t}.
$$
Note that in this case the specific form of $b_{\varepsilon}(x, \mu)$ has enabled us to use
the exponential convergence of $B(\mu_t)$ to zero in place of Condition~(H3) (Lipschitzness in~$\mu$).

The case where  $\varepsilon=1$ has been considered in the introduction.
In addition, it is covered by Theorem~\ref{th1}.

We now consider the case where $\varepsilon>1$.
Then the unique stationary solution is the standard Gaussian measure and the
solutions to the  Cauchy problem converge in total variation norm to this measure  only
if the initial condition $\nu$ has zero mean. In case of a nonzero mean of $\nu$ it is
 easy to see that there is no  convergence in total variation norm with weight
 $(1+|x|)$, because in this case the means of  $\mu_t$ must converge to the mean
of the stationary distribution, while in our example with $\varepsilon>1$ the mean of $\mu_t$
equals $B(\nu)e^{(\varepsilon-1)t}$ and tends to infinity.
However, one still might hope  that there is convergence in some weaker sense,
for  example, weak convergence. However, weak convergence also fails in case of a nonzero mean of~$\nu$.
Indeed, let $\varphi\in C_0^{\infty}(\mathbb{R})$. Then
$$
\int\varphi\,d\mu_t-\int\varphi\,d\nu=\int_0^t
\Bigl[\int(\varphi''-x\varphi')\,d\mu_s+
B(\nu)e^{(\varepsilon-1)s}\int\varphi'\,d\mu_s\Bigr]\,ds.
$$
If $\mu_t$ converges to the stationary distribution  $\mu$ in the sense
of weak convergence, then
all integrals with $\varphi''$, $x\varphi'$, $\varphi'$ and $\varphi$
converge to some constants.
Let $$\int\varphi'\,d\mu\neq 0.$$
 Then the right-hand side of the integral equality above is unbounded as
 $t\to\infty$, which contradicts the boundedness of the left-hand side.
\end{remark}

\begin{remark}\label{remK}
\rm
Let us explain why convergence in the Kantorovich distance can be more easily verified.
We consider the following example (see also \cite[Remark~4.2]{BKSH17}).
Let $A=I$, $\sup_{x, \mu}|b(x, \mu)|(1+|x|)^{-m}<\infty$
and suppose that there exist numbers $\kappa>0$ and $C>0$ such that
$$
|b(x, \mu)-b(x, \sigma)|\le CW_1(\mu, \sigma), \quad
\langle x-y, b(x, \mu)-b(y, \mu)\rangle\le-\kappa|x-y|^2,
$$
where $W_1(\mu, \sigma)$ is the Kantorovich metric defined as the supremum of the
quantities
$$
\int_{\mathbb{R}^d} \varphi\,d(\mu-\sigma)
$$
over all $1$-Lipschitz functions $\varphi$.
Suppose that $C<\kappa$.
Then it is not difficult to show that the solution  $\{\mu_t\}$ to the  Cauchy problem (\ref{eq1A})
with the initial condition $\nu$
converges to the stationary solution~$\mu$  (the existence of which
follows by \cite[Theorem 4.1]{BKSH17}),
moreover,
$$W_1(\mu_t, \mu)\le e^{-(\kappa-C)t}W_1(\nu, \mu).$$
Let us note the remarkable sharpness of this result: in Example~\ref{ex1.1} we have
\mbox{$\kappa=1$,} $C=\varepsilon$, and already for $C=\kappa$ ($\varepsilon=1$) the  assertion about
convergence fails. A~justification of this
result is not difficult. Let $\{T_t\}$ be the semigroup on $L^1(\mu)$ generated
by the operator $Lf=\Delta f+\langle b(x, \mu), \nabla f\rangle$ (see \cite[Chapter~5]{BKRSH-book}).
The measure $\mu$ is invariant  with respect to~$T_t$.
For every function $f\in C_0^{\infty}(\mathbb{R}^d)$
with $|\nabla f|\le 1$ we have  $|\nabla T_tf|\le e^{-\kappa t}$
(see \cite[Theorem~5.6.41]{BKRSH-book}).
Multiplying the  Fokker--Planck--Kolmogorov  equation by the function
$u(x, t)=T_{\tau-t}f(x)$ and integrating by parts (which is possible by the stated properties of~$u$) we derive that
$$
\int_{\mathbb{R}^d} f\,d(\mu_{\tau}-\mu)=\int_{\mathbb{R}^d} u(x, 0)\,d(\nu-\mu)+
\int_0^{\tau}\int_{\mathbb{R}^d} \langle b(x, \mu_t)-b(x, \mu), \nabla u\rangle\,d\mu_t\,dt,
$$
 which yields the estimate
$$
W_1(\mu_t, \mu)\le W_1(\nu, \mu)e^{-\kappa\tau}
+C\int_0^{\tau}W_1(\mu_t, \mu)e^{-\kappa(\tau-t)}\,dt.
$$
It remains to apply Gronwall's inequality.
\end{remark}

\section{Solvability of nonlinear  Fokker--Planck--Kolmogorov equations}

In this section  we discuss  conditions under which the stationary  equation
and the Cauchy  problem for the   Fokker--Planck--Kolmogorov equation  have solutions.

Let  $\mathcal{H}$ be the set of all functions
$h\in C(\mathbb{R}^d)$  such that
$\sup_x |h(x)|/W(x)<\infty$ and
$$
L_{\sigma}\psi(x)=C_1(\sigma)h(x)+C_2(\sigma)
$$
for some function $\psi\in I_0^W$ and all
$\sigma\in\mathcal{P}_V(\mathbb{R}^d)$,
where $C_1(\sigma)$ and $C_2(\sigma)$ are numbers depending on~$\sigma$.
According to Proposition \ref{pr02}, if $\mu$ is a solution to the
linear stationary Fokker--Planck--Kolmogorov  equation with the drift coefficient
$b(x, \sigma)$, then $\mu(h)=\sigma(h)$ for every function $h\in\mathcal{H}$.

\begin{proposition}\label{pr1}
Suppose that for every measure $\nu\in\mathcal{P}_V(\mathbb{R}^d)$ and every ball $U$
there exist numbers $\delta\in(0, 1)$, $C>0$, $\Lambda>0$ and for every ball $U$
there exists a positive function $\alpha\mapsto N(\alpha, U)$, which
 depends also on~$\nu$,
such that for all $\alpha>0$ and $\mu\in \mathcal{M}_{\alpha}(V)$
 satisfying the condition
$\mu|_{\mathcal{H}}=\nu|_{\mathcal{H}}$ we have
$$
L_{\mu}V\le (1-\delta)C+\Lambda(\delta\alpha-V), \quad \sup_{x\in U}|b(x, \mu)|
\le N(\alpha, U).
$$
Suppose also that for all $\alpha>0$ and $\mu_n, \mu\in\mathcal{M}_{\alpha}(V)$
convergence $\|\mu_n-\mu\|_W\to 0$ yields that
$b(x, \mu_n)$ converges to $b(x, \mu)$ uniformly on every ball.
Then for every measure $\nu\in\mathcal{P}_V(\mathbb{R}^d)$ there
exists a solution $\mu$ to the stationary equation {\rm (\ref{eq1AS})} such
that $\nu|_{\mathcal{H}}=\mu|_{\mathcal{H}}$.
\end{proposition}
\begin{proof}
Let $\sigma\in\mathcal{P}_V(\mathbb{R}^d)$ and
$\sigma|_{\mathcal{H}}=\nu|_{\mathcal{H}}$.
It is well-known (see \cite[Corollary 2.4.2 and
Theorem 4.1.6]{BKRSH-book}) that
there is a unique probability solution $\mu$
to the linear equation
$$
L_\sigma^{*}\mu=0.
$$
According to \cite[Theorem 2.3.2]{BKRSH-book} we have
$$
\int_{\mathbb{R}^d}V\,d\mu\le (1-\delta)C\Lambda^{-1}+\delta
\int_{\mathbb{R}^d}V\,d\sigma.
$$
Set
$$
\alpha=\max\{C\Lambda^{-1}, \|V\|_{L^1(\sigma)}\}.
$$
Thus we have $\|V\|_{L^1(\mu)}\le \alpha$ and $\mu|_{\mathcal{H}}
=\sigma|_{\mathcal{H}}=\nu|_{\mathcal{H}}$.
Since $\sup_{U}|b(x, \mu)|\le N(\alpha, U)$,
the measure $\mu$ has a density $\varrho$ and,
for every ball $U$ and for some $\delta\in(0, 1)$, one has
$$
\|\varrho\|_{C^{\delta}(U)}\le C(U),
$$
where $C(U)$ depends only on $U$, $d$, $N(\alpha, U)$, and $A$
(see \cite[Corollary 1.6.7]{BKRSH-book}).
Set
$$
\mathcal{K}=\Bigl\{\mu\in\mathcal{M}_{\alpha}(V)\colon
\mu|_{\mathcal{H}}=\nu|_{\mathcal{H}}, \ \
\mu=\varrho\,dx, \ \ \|\varrho\|_{C^{\delta}(U)}\le C(U)\Bigr\},
$$
where $C^{\delta}(U)$ is the space of $\delta$-H\"older functions
with its natural H\"older norm
$$
\|g\|_{C^{\delta}(U)}
=\sup_x|g(x)|+\sup_{x\not=y}|g(x)-g(y)|/|x-y|^\delta.
$$
The set $\mathcal{K}$ is convex and
compact in $L^1(\mathbb{R}^d)$ and  $T\colon \mathcal{K}\to\mathcal{K}$.
Compactness follows from the fact that, for any sequence $\{\varrho_n\}
\in \mathcal{K}$, the measures
$\varrho_n\, dx$ are uniformly tight, hence there is a weakly
convergent subsequence. The uniform local
H\"older continuity enables us to select a further subsequence
in $\{\varrho_n\}$ that converges uniformly
on balls. Along with weak convergence this yields convergence in variation.

Next, if $\varrho_n\in\mathcal{K}$ and $\varrho_n\to\varrho$
in $L^1(\mathbb{R}^d)$, then
$$
\|\varrho_n-\varrho\|_{W}\le \max_{|x|\le R}W(x)\|\varrho_{n}
-\varrho\|_{L^1(\mathbb{R}^d)}
+2\alpha(\min_{|x|\ge R}W(x))^{-1}.
$$
It follows that $\|\varrho_n-\varrho\|_{W}\to 0$
and $b(x, \varrho_n\,dx)\to b(x, \varrho\,dx)$ uniformly on every ball.
For every $h\in\mathcal{H}$ we obtain $\displaystyle
\int h\varrho_n\,dx\to \int h\varrho\,dx$.
Moreover, $\sigma_n=T(\varrho_n)$ has a subsequence that
converges uniformly on every ball and in $L^1(\mathbb{R}^d)$.
Therefore, $\sigma_n\to\sigma$,
where $\sigma$ is a unique probability solution to the equation
$$
L_{\varrho\,dx}^{*}\sigma=0.
$$
Consequently, $T$ is a continuous mapping.
By Schauder's fixed point theorem there exists $\mu\in\mathcal{K}$
such that $T(\mu)=\mu$.
\end{proof}

\begin{example}\label{ex-st}\rm
Let $d\ge 1$, $A=I$,
$$
b(x, \mu)=-Rx
+\biggl(\int_{\mathbb{R}^d} \langle v, y\rangle\,\mu(dy)\biggr)h
+\int_{\mathbb{R}^d} H(x, y)\,\mu(dy),
$$
where $R$ is a constant matrix, $v$ and $h$ are constant vectors and
$$R^{*}v=\lambda v, \quad \langle v, h\rangle=\lambda, \quad \langle H(x, y), v\rangle=0.$$
Suppose also that
$$
\langle Rx, x\rangle \ge q|x|^2, \quad q>0, \quad \sup_{x, y}|H(x, y)|<\infty.
$$
We show that all conditions in Proposition \ref{pr1} are fulfilled with $V(x)=(1+|x|)^2$ and $W(x)=(1+|x|)$.
The function $x\to \langle v, x\rangle$ belongs to~$I_0^W$.
In addition, according to Example \ref{exa1} the function $x\to \langle v, x\rangle$ belongs to $\mathcal{H}$.
Let $\mu$ be a  probability measure with
$$
\int_{\mathbb{R}^d} \langle v, y\rangle\,\mu(dy)=Q.
$$
Then
$$
L_{\mu}(1+|x|^2)\le 2d+q+q^{-1}(|h||Q|+\sup_{x, y}|H(x, y)|)^2-q(1+|x|^2),
$$
$$
|b(x, \mu)|\le (\|R\|+|h||Q|+\sup_{x, y}|H(x, y)|)(1+|x|^2)^{1/2}.
$$
Thus, the conditions of the theorem are fulfilled and for every number  $Q$ there
exists a solution to the stationary equation with such coefficients.
\end{example}

\begin{remark}\label{rem1}\rm
Suppose that in place of  conditions of Proposition \ref{pr1}  the following
stronger conditions hold: there are positive numbers $C_1$, $C_2$ such that,
for every ball $U$, there is a positive function
$\alpha\mapsto N(\alpha, U)$
such that for all $\alpha>0$ and all $\mu\in \mathcal{M}_{\alpha}(V)$
we have
$$
L_{\mu}V\le C_1-C_2V, \quad \sup_{x\in U}|b(x, \mu)|\le N(\alpha, U).
$$
Suppose also that for all $\alpha>0$ and $\mu_n,
 \mu\in\mathcal{M}_{\alpha}(V)$
convergence  $\|\mu_n-\mu\|_W\to 0$ yields  that
$b(x, \mu_n)$ converges to $b(x, \mu)$ uniformly on every ball.
Then there exists a stationary solution $\mu$ such that
$$
\int_{\mathbb{R}^d}V\,d\mu\le\frac{C_1}{C_2}.
$$
The proof repeats the reasoning given above.
We only observe that for every $\sigma\in\mathcal{P}_V(\mathbb{R}^d)$
the solution $\mu$ to the equation  $L_{\sigma}^{*}\mu=0$
satisfies the  inequality $\|V\|_{L^1(\mu)}<C_1/C_2$ (see
\cite[Theorem 2.3.2]{BKRSH-book})
and the mapping $T$ from the proof of Proposition \ref{pr1}
maps $\mathcal{M}_{\alpha}(V)$ with $\alpha=C_1/C_2$ into the same set.
\end{remark}

We now discuss the  Cauchy problem (\ref{eq1A}). The next proposition gives
conditions that guarantee the existence of a solution $\mu_t$ on $[0, +\infty)$ such that for every $T>0$
one has
$$
\sup_{t\in[0, T]}\int_{\mathbb{R}^d}V\,d\mu_t<\infty.
$$

\begin{proposition}\label{pr2}
Suppose that there exist continuous positive functions
$N_1$, $N_2$, $N_3$ such that
for all $\alpha>0$ and $\mu, \sigma\in\mathcal{M}_{\alpha}(V)$
we have
$$
L_{\mu}V(x)\le N_1(\alpha), \quad |b(x, \mu)|\le N_2(\alpha)
V^{\frac{1}{2}-\gamma}, \quad
|b(x, \mu)-b(x, \sigma)|\le N_3(\alpha)V(x)^{\frac{1}{2}}\|\mu-\sigma\|_{W}.
$$
If
\begin{equation}\label{b4}
\int_0^{+\infty}\frac{d\alpha}{N_1(\alpha)}=+\infty,
\end{equation}
then for every initial condition $\nu\in\mathcal{P}_V(\mathbb{R}^d)$ there
exists a solution  $\{\mu_t\}$ to the  Cauchy problem {\rm (\ref{eq1A})} on $[0, +\infty)$ such that for every $T>0$
$$
\sup_{t\in[0, T]}\int_{\mathbb{R}^d}V\,d\mu_t<\infty.
$$
\end{proposition}
\begin{proof}
We first prove the existence of a solution on $[0, T]$ for every fixed $T>0$.
Let $\alpha(t)$ be a positive continuous function on $[0, T]$. If
$\sigma_t\in\mathcal{M}_{\alpha(t)}(V)$, then
$$
|b(x, \sigma_t)|\le MV(x)^{1/2-\gamma}, \quad M=\max_{t\in [0, T]}N_2(\alpha(t)).
$$
Since $L_{\sigma_t}V(x)\le N_1(\alpha(t))\le \max_{t\in [0, T]}N_1(\alpha(t))$,
by the standard existence condition involving a Lyapunov function (see \cite{BKRSH-book})
there exists a unique solution $\{\mu_t\}$ to the Cauchy problem
$$
\partial_t\mu_t=L_{\sigma_t}^{*}\mu_t, \quad \mu_0=\nu.
$$
Moreover, $\mu_t(dx)=\varrho(x, t)\,dx$ and for some $\delta\in(0, 1)$ we have
$$
\|\varrho\|_{C^{\delta}(U\times J)}\le C(U, J)
$$
for every ball $U\subset\mathbb{R}^d$ and every interval $J\subset(0, T)$.
Note also that by  \cite[Theorem 7.1.1]{BKRSH-book}
$$
\int_{\mathbb{R}^d}V(x)\,\mu_t(dx)\le
\int_0^t N_1(\alpha(s))\,ds+\int_{\mathbb{R}^d}V(x)\,\nu(dx).
$$
Let us define $\alpha(t)$ by means of the following expression:
$$
\int_{\alpha_0}^{\alpha(t)}\frac{du}{N_1(u)}=t, \quad \alpha_0
=\int_{\mathbb{R}^d}V(x)\,\nu(dx).
$$
By (\ref{b4}) the function $\alpha$ is defined on $[0, +\infty)$.
If we take $\sigma_t\in\mathcal{M}_{\alpha(t)}(V)$,
then the corresponding solution $\mu_t$ will belong
to $\mathcal{M}_{\alpha(t)}(V)$.
Let $\mathcal{K}_1$ be the set of all functions $\varrho$
on $\mathbb{R}^d\times[0, T]$
such that $\varrho\in C(\mathbb{R}^d\times(0, T))$, $\varrho\ge 0$ and
$$
\int_{\mathbb{R}^d}\varrho(x, t)\,dx=1, \quad
\int_{\mathbb{R}^d}\varrho(x, t)V(x)\,dx\le\alpha(t), \quad
\|\varrho\|_{C^{\delta}(U\times J)}\le C(U, J).
$$
Note that $\mathcal{K}_1$ is convex and compact
in $L^1(\mathbb{R}^d\times[0, T])$.
Let us define $T\colon \mathcal{K}_1\to\mathcal{K}_1$ as follows:
to each $\sigma_t=v(x, t)\,dx$ with $v\in\mathcal{K}_1$,
the mapping $T$ associates the solution $\mu_t=\varrho(x, t)\,dx$.
If $v_n, v\in \mathcal{K}_1$ and $v_n\to v$
in $L^1(\mathbb{R}^d\times[0, T])$, then
$\|v_n(y, t)dy-v(y, t)dy\|_{W}\to 0$. It follows
that $b(x, v_n(y, t)\,dy)\to b(x, v(y, t)\,dy)$
for all $(x, t)$. Thus, the corresponding
solutions $\varrho_n$ converge to the solution $\varrho$, hence
the mapping $T$ is continuous. By Schauder's
fixed point theorem there exists $\varrho\in\mathcal{K}_1$
such that $T(\varrho)=\varrho$.
By  \cite[Theorem 3.1]{ManRomShap} the constructed solution
to the Cauchy problem (\ref{eq1A})
is unique on $[0, T]$ for every $T>0$.
This yields the existence and uniqueness on the whole half-line
$[0, +\infty)$.
\end{proof}

This paper has been supported by the RFBR Grant 17-01-00662, the RF President Grant MD-207.2017.1,
the Simons Foundation,  and the DFG through SFB 1283 at Bielefeld University.

\end{document}